\documentclass{amsart}

\usepackage[T1]{fontenc}
\usepackage[utf8]{inputenc}

\usepackage{amssymb,amsmath,amsfonts,amsthm,mathrsfs,mathtools,amscd,tikz-cd,dirtytalk,comment}

\newtheorem{theorem}{Theorem}[section]
\newtheorem{lemma}[theorem]{Lemma}
\newtheorem{proposition}[theorem]{Proposition}
\newtheorem{corollary}[theorem]{Corollary}

\theoremstyle{definition}

\newtheorem{example}[theorem]{Example}

\theoremstyle{remark}
\newtheorem{remark}[theorem]{Remark}
\newtheorem{notation}[theorem]{Notation}

\numberwithin{equation}{section}

\DeclareMathOperator{\Aut}{Aut}

\DeclareMathOperator{\End}{End}
\DeclareMathOperator{\Fix}{Fix}

\DeclareMathOperator{\id}{id}

\DeclareMathOperator{\Inn}{Inn}

\DeclareMathOperator{\Perm}{Perm}

\newcommand{\ep}{\varepsilon}

\newcommand{\ka}{\kappa}

\newcommand{\si}{\sigma}
\newcommand{\ta}{\tau}

\newcommand{\ph}{\varphi}

\newcommand{\cK}{\mathcal{K}}
\newcommand{\cL}{\mathcal{L}}

\newcommand{\op}{\mathrm{op}}

\newcommand{\lexp}[2]{\null^{#2} \mkern-2mu #1}
\newcommand{\lexpp}[2]{\null^{#2} \mkern-2mu (#1)}

\usepackage[hidelinks]{hyperref}

\begin{document}

\title[Hopf--Galois structures and skew braces]{On the connection between Hopf--Galois structures and skew braces}

\author{L.~Stefanello}

\address{Department of Mathematics,
          Universit\`a di Pisa,
          Largo Bruno Pontecorvo 5, 56127 Pisa, Italy}
        
\email{lorenzo.stefanello@phd.unipi.it}

\urladdr{https://people.dm.unipi.it/stefanello/}
\thanks{The first author was a member of GNSAGA (INdAM)}

\author{S.~Trappeniers}

\address{Department of Mathematics and Data Science, 
 Vrije Universiteit Brussel, 
 Pleinlaan 2, 
 1050 Brussels, Belgium} 

\email{senne.trappeniers@vub.be} 
\thanks{The second author was supported by Fonds voor Wetenschappelijk Onderzoek (Flanders), via an FWO Aspirant-fellowship, grant 1160522N}



\begin{abstract}
We present a different version of the well-known connection between Hopf--Galois structures and skew braces, building on a recent paper of A.~Koch and P.~J.~Truman. 

We show that the known results that involve this connection easily carry over to this new perspective, and that new ones naturally appear.

As an application, we present new insights on the study of the surjectivity of the Hopf--Galois correspondence, explaining in more detail the role of bi-skew braces in Hopf--Galois theory.
\end{abstract}

\maketitle

\section{Introduction}

Let $L/K$ be a finite extension of fields. A Hopf--Galois structure on $L/K$ consists of a $K$-Hopf algebra $H$ together with an action of $H$ on $L$ that satisfies certain technical conditions. When $L/K$ is Galois with Galois group $G$, the prototypical example consists of the group algebra $K[G]$ with the usual Galois action on $L$; indeed, the required properties for a Hopf--Galois structure mimic precisely those of this structure, which is called the classical structure. 

Hopf--Galois theory was initially introduced in the context of purely inseparable extensions by S.~U.~Chase and M.~E.~Sweedler~\cite{CS69}, but after it was mainly studied for separable extensions, providing a generalisation of classical Galois theory. In the particular case in which the extension is also Galois, Hopf--Galois structures have been shown to be extremely useful in dealing with problems in arithmetic. For example, as discussed by N.~P.~Byott in~\cite{Byo97}, there are situations in which the Galois module structure of an extension of $p$-adic fields can be better described in a Hopf--Galois structure different from the classical one; see~\cite{Chi00,CGKKKTU21} for a detailed analysis on the role of Hopf--Galois theory in local Galois module theory. 

A main role in the development of this theory was played by a groundbreaking result of C.~Greither and B.~Pareigis~\cite{GP87}. We assume that $L/K$ is Galois with Galois group $G$, which is the case of interest in the paper, underlining that the result can be stated also for separable non-normal extensions. Then there exists a bijective correspondence between Hopf--Galois structures on $L/K$ and regular subgroups of the permutation group $\Perm(G)$ of $G$ normalised by $\lambda(G)$, the image of $G$ under the left regular representation. For example, $\rho(G)$, the image of $G$ under the right regular representation, corresponds to the classical structure, while $\lambda(G)$ corresponds to the so-called canonical nonclassical structure, different from the classical one when $G$ is not abelian. We define the type of a Hopf--Galois structure to be the isomorphism class of the corresponding regular subgroup of $\Perm(G)$.

This result was followed by new approaches to the theory, and problems of existence and classification have been studied by several authors;  given a group $N$, does there exist a Hopf--Galois structure of type $N$ on $L/K$?  Can we classify and count the Hopf--Galois structures on $L/K$? A precise survey of the main results developed in the last years can be found in~\cite{CGKKKTU21}. 

A deep problem that can be approached with Greither--Pareigis theory regards the surjectivity of the Hopf--Galois correspondence. Given a Hopf--Galois structure on $L/K$ with $K$-Hopf algebra $H$, we can attach to each $K$-sub Hopf algebra of $H$ an intermediate field of $L/K$ in a natural way. The correspondence we get is called the Hopf--Galois correspondence, which can be shown to be injective~\cite{CS69} but not necessarily surjective. For example, if we consider the classical structure, then we recover the usual Galois correspondence, which is surjective. But it was proved in~\cite{GP87} that if we consider the canonical nonclassical structure, then the image of the Hopf--Galois correspondence consists precisely of the normal intermediate fields of $L/K$; this shows that if $G$ is Hamiltonian (that is, nonabelian with all the subgroups normal), then the Hopf--Galois correspondence is surjective, but as soon as the group is not abelian nor Hamiltonian, we find a Hopf--Galois structure for which the Hopf--Galois correspondence is not surjective. 

More generally, given a Hopf--Galois structure on $L/K$ with Hopf algebra $H$ corresponding  to a regular subgroup $N$ of $\Perm(G)$ normalised by $\lambda(G)$, we know that there exists a bijective correspondence between $K$-sub Hopf algebras of $H$ and subgroups of $N$ normalised by $\lambda(G)$; the first explicit proof of this fact can be found in~\cite[Proposition 2.2]{CRV16c}. As there always exists a bijective correspondence between intermediate fields of $L/K$ and subgroups of the Galois group $G$, we can translate the Hopf--Galois correspondence to find a correspondence between subgroups of $N$ normalised by $\lambda(G)$ and subgroups of $G$. This means that for groups of small order the problem can be approached from a quantitative point of view; in~\cite{KKTU19b}, the authors used GAP~\cite{GAP} to deal with groups of order $42$ and found some nonclassical Hopf--Galois structures for which the number of subgroups of the Galois group $G$ equals the number of subgroups of $N$ normalised by $\lambda(G)$, meaning that the Hopf--Galois correspondence for these structures is surjective.

A look in the literature seems to suggest that these cases are not really common. Beside these examples and the aforementioned classical structure and canonical nonclassical structure when $G$ is Hamiltonian, there exists only one other known class of Hopf--Galois structures for which the Hopf--Galois correspondence is surjective; this was obtained from the study of the connection between Hopf--Galois structures and skew braces, objects introduced by L.~Guarnieri and L.~Vendramin~\cite{GV17}, building on the pioneering work of W.~Rump~\cite{Rum07a}. Skew braces are related with several other topics, such as radical rings, solutions of the Yang--Baxter equation, and the holomorph of a group. In particular, they are connected with regular subgroups of permutation groups, and in this way, also with Hopf--Galois structures. This connection, which is not bijective, was initially suggested in~\cite{Bac16} and then made precise in the appendix of Byott and Vendramin in~\cite{SV18}.

Thanks to this connection, the problem of the surjectivity of the Hopf--Galois correspondence was translated into a different language by L.~N.~Childs~\cite{Chi17,Chi18}, who showed that given a Hopf--Galois structure on $L/K$ with Hopf algebra $H$, there exists a bijective correspondence between $K$-sub Hopf algebras of $H$ and certain substructures of the associated skew brace. In this way, Childs proved that for all the Hopf--Galois structures on a Galois extension with Galois group cyclic of odd prime power order, the Hopf--Galois correspondence is surjective. 

Despite this promising start, to the best knowledge of the authors, no further examples of this behaviour have been found. A new approach, introduced in~\cite{Chi21}, seems to suggest how difficult it is to find them. Namely, instead of looking for Hopf--Galois structures for which the Hopf--Galois correspondence is surjective, one can study the failure of the surjectivity. Given a Hopf--Galois structure on $L/K$ with Hopf algebra $H$, how far is the Hopf--Galois correspondence from being surjective? The idea is to compute (or estimate) the ratio between the number of $K$-sub Hopf algebras of $H$ and the number of intermediate fields of $L/K$, which was translated by Childs in a problem regarding just the associated skew brace.

 A possible explanation for the lack of new examples could be given by the fact that the substructures of skew braces studied by Childs, which seem to arise naturally from Hopf--Galois theory, are not the usual substructures considered in the theory of the skew braces, namely, left ideals, strong left ideals, and ideals. This issue was initially addressed by A.~Koch and P.~J.~Truman~\cite{KT20}, who considered the notion of opposite skew brace and showed that the substructures studied by Childs coincide with left ideals of the opposite skew brace. They moved the problem to a more familiar setting, and combined this observation with the results of~\cite{KKTU19b} to describe some known properties of Hopf--Galois structures in terms of the opposite skew brace. 

This intuition is at the very base of this paper, where we present a new version of the known connection between Hopf--Galois structures and skew braces, as per the following points:
\begin{enumerate}
	\item Use directly the opposite skew brace.
	\item Make the connection bijective. 
	\item Forget about the regular subgroup.
\end{enumerate}
The idea is that using this new point of view one can explicitly see how the knowledge of the structure of a skew brace gives useful and qualitative information for the associated Hopf--Galois structure. In particular, the role of bi-skew braces, certain skew braces introduced by Childs~\cite{Chi19} and then further studied by A.~Caranti~\cite{Car20} and the authors~\cite{ST23}, seems to appear in a more transparent way from this new perspective, for example in order to find Hopf--Galois structures for which the Hopf--Galois correspondence is surjective.

The paper is organised as follows. In section~\ref{sec: hopfgalois skewbraces}, we introduce the necessary preliminaries on Hopf--Galois structures, skew braces, and their connections. We also recall the tool of Galois descent, useful throughout the rest of the paper. In section~\ref{sec: main}, we explicitly describe the new connection we propose. We remark how the known advantages of the usual connection still apply in the new perspective, and we see how some old and new results can be explained and derived.
In section~\ref{sec: examples}, we use the new point of view to deal with the Hopf--Galois correspondence. In particular, we present new qualitative results, examples, and statements to explain why, in some situations, the Hopf--Galois correspondence is surjective, from a more general perspective. A main role here is played by bi-skew braces.

\section{Preliminaries}\label{sec: hopfgalois skewbraces}

\subsection{Hopf--Galois structures}

Let $L/K$ be a finite extension of fields. A \emph{Hopf--Galois structure} on $L/K$ consists of a $K$-Hopf algebra $H$ together with an action $\star$ of $H$ on $L$ such that $L$ is an $H$-module algebra and the $K$-linear map
			\begin{equation*}
				L\otimes_K H\to \End_K(L),\quad x\otimes h\mapsto (y\mapsto x(h\star y))
			\end{equation*}
			is bijective. (We remark that two isomorphic $K$-Hopf algebras whose action on $L$ respect the isomorphism give the same structure). For more insights on the definition, we refer to~\cite{Chi00}.
					
For example, when $L/K$ is Galois with Galois group $G$, the \emph{classical structure} consists of the group algebra $K[G]$ together with the usual Galois action. 

Following~\cite{CS69}, given a Hopf--Galois structure on $L/K$ with $K$-Hopf algebra $H$, we can attach to each $K$-sub Hopf algebra $H'$ of $H$ an intermediate field $F$ of $L/K$, as follows:
\begin{equation*}
	F=L^{H'}=\{x\in L\mid h'\star x=\ep(h')x \text{ for all $h'\in H'$}\},
\end{equation*}
where $\ep$ denotes the counit of $H'$. We obtain in this way the \emph{Hopf--Galois correspondence}, which is always injective. We remark that the $F$-Hopf algebra $F\otimes_K H'$ acts on $L$ naturally and gives a Hopf--Galois structure on $L/F$, and in particular, $[L:F]$ equals the dimension of $H'$ as $K$-vector space; see also~\cite[section 7]{CGKKKTU21} for more details.

A $K$-sub Hopf algebra $H'$ of $H$ is \emph{normal} if for all $h\in H$ and $h'\in H'$,
\begin{align*}
\sum_{(h)}h_{(1)}h'S(h_{(2)})\in H',\quad 	\sum_{(h)}S(h_{(1)})h'h_{(2)}\in H',
\end{align*}
where, in Sweedler's notation, the image of $h$ under the comultiplication $\Delta$ of $H$ is $\Delta(h)=\sum_{(h)}h_{(1)}\otimes h_{(2)}$, and $S$ denotes the antipode of $H$.
If $H'$ is a normal $K$-sub Hopf algebra of $H$, then by~\cite[Lemma 3.4.2 and Proposition 3.4.3]{Mon93} there exists a short exact sequence
\begin{equation*}
	K\to H'\to H\to \overline{H}\to K
\end{equation*}
of $K$-Hopf algebras, in the sense of~\cite[Proposition 1.2.3]{AD95b}. Moreover, if $F=L^{H'}$, then the action of $H$ on $L$ yields an action of $\overline{H}$ on $F$ which gives a Hopf--Galois structure on $F/K$; see~\cite[Lemma 4.1]{Byo02}.

We recall that $h\in H$ is a \emph{grouplike element} if $\Delta(h)=h\otimes h$.

One fundamental tool in this theory is given by Galois descent; we briefly recall it here for the convenience of the reader, summarising~\cite[section 2.12]{Chi00}. Suppose that $L/K$ is Galois with Galois group $G$. Given an $L$-Hopf algebra $M$ on which $G$ acts semilinearly, we say that $M$ is \emph{$G$-compatible} if all the maps defining the structure of $M$ as an $L$-Hopf algebra are $G$-equivariant. (Here $G$ acts on $L$ via Galois action and on $M\otimes_L M$ diagonally.) 

Denote by $\cK$ the category of $K$-Hopf algebras, where morphisms are $K$-Hopf algebra homomorphisms, and by $\cL$ the category of $G$-compatible $L$-Hopf algebras, where morphisms are $G$-equivariant $L$-Hopf algebra homomorphisms. Then there exists an equivalence of categories between $\cK$ and $\cL$, as follows:
\begin{itemize}
	\item If $H\in \cK$, then $L\otimes_K H\in \cL$, where here $G$ acts on the first factor of the tensor product; given a morphism $\ph\colon H_1\to H_2$ in $\cK$, we have that $\id\otimes\ph\colon L\otimes_K H_1\to L\otimes_K H_2$
		is a morphism in $\cL$.
	\item If $M\in \cL$, then $M^G=\{m\in M\mid G\text{ acts trivially on $m$}\}\in\cK$; given a morphism $\psi\colon M_1\to M_2$ in $\cL$, the restriction of $\psi$ to $M_1^G$ is a morphism $M_1^G\to M_2^G$. 
	\item If $H\in\cK$, then
\begin{equation*}
	H\to (L\otimes_K H)^G,\quad h\mapsto 1\otimes h
\end{equation*}
 is an isomorphism in $\cK$, and if $M\in \cL$, then
 \begin{equation*}
 	L\otimes_K M^G\to M,\quad l\otimes m\mapsto lm
 \end{equation*}
 is an isomorphism in $\cL$.
\end{itemize} 
 We immediately derive some consequences:
 \begin{itemize}
 	\item Let $M\in \cL$. Then there exists a bijective correspondence between $K$-sub Hopf algebras of $M^G$ and $L$-sub Hopf algebras of $M$ which are invariant under the action of $G$ on $M$. Explicitly, given such an $L$-sub Hopf algebra $M'$, the corresponding $K$-sub Hopf algebra is $(M')^G$, and $M'$ is normal in $M$ if and only if $(M')^G$ is normal in $M^G$.  
 	\item Let
 		\begin{equation*}
	L\to A\to M \to B\to L
\end{equation*}
be a short exact sequence of $L$-Hopf algebras. If all the $L$-Hopf algebras are $G$-compatible and all the $L$-Hopf algebra homomorphisms are $G$-equivariant, then
 \begin{equation*}
	K\to A^G\to M^G \to B^G\to K
\end{equation*}
is a short exact sequence of $K$-Hopf algebras. 
\item For all $M_1,M_2\in \cL$, we have that $(M_1\otimes_L M_2)^G$ and  $M_1^G\otimes_K M_2^G$ are isomorphic as $K$-Hopf algebras.
\item Let $M\in \cL$, and take $h\in M^G$. Then $h$ is a grouplike element of $M^G$ if and only if $h$ is a grouplike element of $M$.
 \end{itemize}

\begin{example}
	Let $N$ be a finite group on which $G$ acts via automorphisms, and extend this to an action of $G$ on $L[N]$, where $G$ acts on $L$ via Galois action. Then it is straightforward to check that $L[N]\in\cL$.
	Here the $L$-sub Hopf algebras of $L[N]$ are the group algebras $L[N']$ for subgroups $N'$ of $N$ (see~\cite[Proposition 2.1]{CRV16c}), and almost by definition, $L[N']$ is normal in $L[N]$ if and only if $N'$ is normal in $N$. We deduce that the $K$-sub Hopf algebras of $L[N]^G$ are of the form $L[N']^G$ for subgroups $N'$ of $N$ invariant under the action of $G$, and $L[N']^G$ is normal in $L[N]^G$ if and only if $N'$ is normal in $N$. Note that moreover, the lattices of $K$-sub Hopf algebras of $L[N]^G$ and subgroups of $N$ invariant under the action of $G$, with the usual binary operations, are isomorphic.
	 
	 If $N'$ is a normal subgroup of $N$ invariant under the action of $G$, then, by~\cite[Proposition 4.14]{Chi00}, 
	 \begin{equation*}
	L\to L[N']\to L[N] \to L[N/N']\to L
\end{equation*} is a short exact sequence of $L$-Hopf algebras which are $G$-compatible, where all the $L$-Hopf algebra homomorphisms are $G$-equivariant, so
	  \begin{equation*}
	K\to L[N']^G\to L[N]^G \to L[N/N']^G\to K
\end{equation*}
is a short exact sequence of $K$-Hopf algebras.

Finally, as the grouplike elements of $L[N]$ are the elements of $N$, we find that the grouplike elements of $L[N]^G$ are the elements of $N$ on which $G$ acts trivially. 
\end{example}

We conclude by mentioning~\cite[Theorem 2.1]{GP87}, whose proof heavily relies on Galois descent, which gives a description of the $K$-Hopf algebras arising in this theory.  Recall that a subgroup $N$ of $\Perm(G)$ is \emph{regular} if the map
\begin{equation*}
	N\to G,\quad \eta\mapsto \eta[1]
\end{equation*}
is a bijection. For example, $\lambda(G)$ and $\rho(G)$ are regular subgroups of $\Perm(G)$, where for all $\sigma,\tau\in G$,
\begin{align*}
	\lambda(\sigma)[\tau]=\sigma \tau,\quad \rho(\sigma)[\tau]=\tau\sigma^{-1}.
\end{align*}

Then there exists a bijective correspondence between Hopf--Galois structures on $L/K$ and regular subgroups of $\Perm(G)$ normalised by $\lambda(G)$; explicitly, if $N$ is such a subgroup, then we can consider the $L$-Hopf algebra $L[N]$, where $G$ acts on $L$ via Galois action and on $N$ via conjugation by $\lambda(G)$, and then via Galois descent take the $K$-Hopf algebra $L[N]^G$, which gives a Hopf--Galois structure on $L/K$ with the following action on $L$:
\begin{equation*}
	\left(\sum_{\eta\in N}a_{\eta}\eta\right) \star x=\sum_{\eta\in N}a_{\eta} (\eta^{-1}[1])(x). 
\end{equation*} 
As already mentioned, $\rho(G)$ corresponds to the classical structure, while $\lambda(G)$ corresponds to the so-called \emph{canonical nonclassical structure}. We say that the \emph{type} of a Hopf--Galois structure is the isomorphism class of the corresponding regular subgroup $N$ of $\Perm(G)$.

\subsection{Skew braces}
A \emph{skew (left) brace} is a triple $(G,\cdot,\circ)$, where $(G,\cdot)$ and $(G,\circ)$ are groups and the following property holds: for all $\si,\ta,\ka\in G$,
	\begin{equation*}
	\si\circ(\ta\cdot \ka)=(\sigma\circ \ta)\cdot \si^{-1}\cdot (\si\circ \ka),
\end{equation*}
where $\si^{-1}$ denotes the inverse of $\si$ in $(G,\cdot)$. We denote by $\overline{\si}$ the inverse of $\si$ with respect to $(G,\circ)$. It is easy to prove that for a skew brace $(G,\cdot,\circ)$, the identities of $(G,\cdot)$ and $(G,\circ)$ coincide. The \emph{order} of a skew brace is the cardinality of the underlying set $G$.  

For example, given a group $(G,\circ)$, we have that $(G,\circ,\circ)$ is a skew brace, which is said to be \emph{trivial}.  Similarly, if we define $\si\circ_{\op} \ta=\ta\circ \si$, then $(G,\circ_{\op},\circ)$ is a skew brace, which is said to be \emph{almost trivial}.
More generally, if $(G,\cdot,\circ)$ is a skew brace, then also $(G,\cdot_{\op},\circ)$ is a skew brace, called the \emph{opposite skew brace} of $(G,\cdot,\circ)$.
\begin{notation}
    Given a skew brace $(G,\cdot,\circ)$, if we want to apply a group theoretical construction with respect to one of the group operations, then we write  the operation as subscript, to avoid ambiguity. For example, we write $\iota_{\cdot}(\si)$ to denote conjugation by $\si$ in $(G,\cdot)$, for $\si\in G$.
\end{notation}

With each element $\si$ of a skew brace $(G,\cdot,\circ)$ we can associate the bijective map 
\begin{equation*}
    \gamma(\si)\colon G\to G, \quad \ta\mapsto \lexp{\ta}{\gamma(\si)}=\si^{-1}\cdot (\si\circ \ta).
\end{equation*}
This yields a group homomorphism
\begin{equation*}
    \gamma\colon (G,\circ)\to \Aut(G,\cdot);
\end{equation*} see~\cite[Proposition 1.9]{GV17}. The map $\gamma$, which is called the called the \emph{gamma function} of $(G,\cdot,\circ)$, gives an action of $(G,\circ)$ on $(G,\cdot)$ via automorphisms.  

For example, the gamma function of $(G,\cdot,\cdot)$ is given by $\gamma(\si)=\id$. Also, if $\gamma$ is the gamma function of a skew brace $(G,\cdot,\circ)$, then the gamma function of $(G,\cdot_{\op},\circ)$ is given by $\iota_{\cdot}(\si)\gamma(\si)$, as an easy computation shows.

Consider two skew braces $(G_1,\cdot,\circ)$ and $(G_2,\cdot,\circ)$. A \emph{skew brace homomorphism} is a map $f\colon G_1\to G_2$ such that $f(\si\cdot \ta)=f(\si)\cdot f(\ta)$ and $f(\si\circ \ta)=f(\si)\circ f(\ta)$ for all $\si,\ta\in G_1$. \emph{Skew brace isomorphisms} and \emph{automorphisms} are defined accordingly, and we denote by $\Aut(G,\cdot,\circ)$ the group of skew brace automorphisms of $(G,\cdot,\circ)$. 

Let $(G,\cdot,\circ)$ be a skew brace. A \emph{left ideal} of $(G,\cdot,\circ)$ is a subgroup $G'$ of $(G,\cdot)$ that is invariant under the action of $(G,\circ)$ via the gamma function $\gamma$ of $(G,\cdot,\circ)$. Note that this immediately implies that $G'$ is also a subgroup of $(G,\circ)$, so $(G',\cdot,\circ)$ is a skew brace, and that actually we can also replace \say{subgroup of $(G,\cdot)$} with \say{subgroup of $(G,\circ)$} in the definition. If $G'$ is normal in $(G,\cdot)$, then we say that $G'$ is a \emph{strong left ideal}; if $G'$ is also normal in $(G,\circ)$, then we say that $G'$ is an \emph{ideal}. In this last case, the quotient $(G/G',\cdot,\circ)$ is a skew brace in a natural way. 

For example,
\begin{equation*}
	\Fix(G,\cdot,\circ)=\{\ta\in G\mid \lexp{\ta}{\gamma(\si)}=\ta\text{ for all $\si\in G$}\}
\end{equation*} is a left ideal of $(G,\cdot,\circ)$; see~\cite[Proposition 1.6]{CSV19}. 

It is clear that the characteristic subgroups of $(G,\cdot)$ are strong left ideals of $(G,\cdot,\circ)$. More generally, the strong left ideals of $(G,\cdot,\circ)$ are precisely the left ideals of $(G,\cdot,\circ)$ which are also left ideals of $(G,\cdot_{\op},\circ)$, because of the description of the gamma function of $(G,\cdot_{\op},\circ)$. 

A skew brace $(G,\cdot,\circ)$ is \emph{metatrivial} if there exists an ideal $G'$ of $(G,\cdot,\circ)$ such that $(G',\cdot,\circ)$ and $(G/G',\cdot,\circ)$ are trivial skew braces. For example, by~\cite[Theorem 2.12]{BNY22}, all the skew braces that can be obtained with~\cite[Theorem 6.6]{ST23} are metatrivial. 

Let $(G_1,\cdot,\circ)$ and $(G_2,\cdot,\circ)$ be skew braces. Following~\cite{SV18}, given a group homomorphism
\begin{equation*}
	\alpha\colon (G_2,\circ)\to \Aut(G_1,\cdot,\circ),
\end{equation*}
we can define a \emph{semidirect product} of $(G_1,\cdot,\circ)$ and $(G_2,\cdot,\circ)$  to be the skew brace $(G,\cdot,\circ)$, where $G=G_1\times G_2$ as set, with $(G,\cdot)=(G_1,\cdot)\times(G_2,\cdot)$ and $(G,\circ)=(G_1,\circ)\rtimes(G_2,\circ)$, where the semidirect product is taken with respect to $\alpha$. When $\alpha$ is the trivial group homomorphism, we find the \emph{direct product} of skew braces, 
\begin{equation*}
    (G_1,\cdot,\circ)\times (G_2,\cdot,\circ).
\end{equation*}
We can generalise the notion of direct product to any finite number of skew braces.  
Note that the gamma function of the direct product of skew braces $(G_i,\cdot,\circ)$ is given by the gamma functions of the skew braces $(G_i,\cdot,\circ)$ in the obvious way.

If $(G,\cdot,\circ)$ is a skew brace isomorphic to a semidirect product of skew braces, then there exist an ideal $G_1$ and a strong left ideal $G_2$ of $(G,\cdot,\circ)$ such that $(G,\circ)$ is the inner semidirect product of $(G_1,\circ)$ and $(G_2,\circ)$, and $(G,\cdot)$ is the inner direct product of $(G_1,\cdot)$ and $(G_2,\cdot)$. When the semidirect product is a direct product, also $G_2$ is an ideal of $(G,\cdot,\circ)$.

Finally, a \emph{bi-skew brace} is a skew brace $(G,\cdot,\circ)$ such that also $(G,\circ,\cdot)$ is a skew brace. If $(G,\cdot,\circ)$ is a bi-skew brace and $\gamma$ is the gamma function of $(G,\cdot,\circ)$, then the gamma function $\gamma'$ of $(G,\circ,\cdot)$ is given by $\gamma'(\si)=\gamma(\si)^{-1}=\gamma(\overline{\si})$; see~\cite[section 3]{Car20}. By~\cite[table at page 1175]{CCDC20}, a skew brace is a bi-skew brace if and only if its gamma function has values in $\Aut(G,\circ)$. If $(G,\cdot,\circ)$ is a bi-skew brace, then the left ideals of $(G,\cdot,\circ)$ and $(G,\circ,\cdot)$ coincide; see~\cite[Lemma 3.1]{ST23}.

\subsection{Hopf--Galois structures and skew braces}

We recall the well-known connection between Hopf--Galois structures and skew braces. While it was originally developed in the appendix of Byott and Vendramin in~\cite{SV18}, we present here an equivalent version, which does not involve explicitly the holomorph, as described in~\cite[Proposition 2.1]{NZ19} (see also~\cite[section 2.8]{CGKKKTU21}). This is based on the following result, which is a slight reformulation of~\cite[Theorem 4.2]{GV17}. 
\begin{theorem}\label{theorem: old connection}
	Let $(G,\cdot)$ and $(G,\circ)$ be groups with the same identity. Then $(G,\cdot,\circ)$ is a skew brace if and only if $\lambda_{\cdot}(G)$ is normalised by $\lambda_{\circ}(G)$ in $\Perm(G)$. 
\end{theorem}

Let $L/K$ be a finite Galois extension of fields with Galois group $(G,\circ)$. 
\begin{itemize}
	\item Consider a Hopf--Galois structure on $L/K$, corresponding to a regular subgroup $N$ of $\Perm(G)$ normalised by $\lambda_{\circ}(G)$. We can use the bijection
\begin{equation*}
	N\to G,\quad \eta\mapsto \eta[1]
\end{equation*}
to transport the group structure of $N$ to $G$. In this way, we find a group structure $(G,\cdot)$ for which it is immediate to show that $\lambda_{\cdot}(G)=N$. By Theorem~\ref{theorem: old connection}, we conclude that $(G,\cdot,\circ)$ is a skew brace.
\item Let $(A,\cdot,\circ)$ be a skew brace with $(A,\circ)\cong (G,\circ)$. Use this bijection to transport the structure of $(A,\cdot)$ to $G$, to find a skew brace $(G,\cdot,\circ)$ isomorphic to $(A,\cdot,\circ)$.  By Theorem~\ref{theorem: old connection}, we have that $N=\lambda_{\cdot}(G)$ is normalised by $\lambda_{\circ}(G)$, so we obtain a Hopf--Galois structure on $L/K$.
\end{itemize}

\begin{example}\label{example: peculiar}
	Peculiarly, under this connection, the classical structure yields the almost trivial skew brace. On the other hand, the trivial skew brace is obtained by the canonical nonclassical structure. 
\end{example}

We immediately state an important and well-known consequence.
\begin{theorem}\label{theorem: quantitative}
	Let $N$ and $G$ be finite groups. Then the following are equivalent:
	\begin{itemize}
		\item There exists a skew brace $(A,\cdot,\circ)$ with $(A,\cdot)\cong N$ and $(A,\circ)\cong G$.
		\item There exists a Hopf--Galois structure of type $N$ on every Galois extension of fields with Galois group isomorphic to $G$. 
	\end{itemize}
\end{theorem}
\begin{remark}
    A bi-skew brace $(A,\cdot,\circ)$ of finite order yields not only a Hopf--Galois structure of type $(A,\cdot)$ on every Galois extension of fields with Galois group isomorphic $(A,\circ)$, but also a Hopf--Galois structure of type $(A,\circ)$ on every Galois extension of fields with Galois group isomorphic to $(A,\cdot)$.
\end{remark}

We underline that the previous connection is not bijective, as distinct Hopf--Galois structures can correspond to isomorphic skew braces. This was precisely quantified in~\cite[Corollary 2.4]{NZ19}; see also~\cite[Corollary 3.1]{KT22} and section~\ref{sec: main}. 
However, there is a way to obtain from this connection a bijective correspondence. Indeed, as a consequence of Theorem~\ref{theorem: old connection} (see~\cite[section 7]{CS22a-p}), given a group $(G,\circ)$, there exists a bijective correspondence between group operations $\cdot$  such that $(G,\cdot,\circ)$ is a skew brace and  regular subgroups of $\Perm(G)$ normalised by $\lambda_{\circ}(G)$, via
\begin{equation*}
	\cdot\mapsto \lambda_{\cdot}(G).
\end{equation*}
In this way, given a finite Galois extension of fields $L/K$ with Galois group $(G,\circ)$, we obtain a bijective correspondence between operations $\cdot$  such that $(G,\cdot,\circ)$ is a skew brace and Hopf--Galois structure on $L/K$, which is a key observation for our new point of view.

\section{The new connection}\label{sec: main}

We begin with our main result, in which we propose a new version of the connection between Hopf--Galois structures and skew braces. We underline that some of the consequences, as developed in this section, can also be obtained from the usual theory, for example from~\cite{KKTU19b}, together with the observations on opposite skew braces in~\cite[Theorem 5.6]{KT20}. However, we prefer to develop directly the theory from this new perspective, to highlight how old and new statements can be derived in a transparent way, without too much effort.

\begin{theorem}\label{theorem: main}
	Let $L/K$ be a finite Galois extension of fields with Galois group $(G,\circ)$. Then the following data are equivalent:
	\begin{itemize}
		\item a Hopf--Galois structure on $L/K$;
		\item an operation $\cdot$ such that $(G,\cdot,\circ)$ is a skew brace. 
	\end{itemize} 
	Explicitly, given an operation $\cdot$ such that $(G,\cdot,\circ)$ is a skew brace, we can consider the Hopf--Galois structure on $L/K$ consisting of the $K$-Hopf algebra $L[G,\cdot]^{(G,\circ)}$, where $(G,\circ)$ acts on $L$ via Galois action and on $(G,\cdot)$ via the gamma function of $(G,\cdot,\circ)$, with action on $L$ given as follows:
	\begin{equation*}
	\left(\sum_{\si\in G}\ell_{\sigma}\si\right) \star x=\sum_{\si\in G}\ell_{\sigma} \si(x). 
\end{equation*}
\end{theorem}

\begin{proof}
Denote by $\mathcal{S}$ the set of group operations $\cdot$ on $G$ such that $(G,\cdot,\circ)$ is a skew brace, and by $\mathcal{R}$ the set of regular subgroups of $\Perm(G)$ normalised by $\lambda_{\circ}(G)$. Consider the composition
\begin{equation*}
    \mathcal{S}\to \mathcal{S}\to \mathcal{R},
\end{equation*}
where the first map is the bijection that sends $\cdot$ to $\cdot_{\op}$ and the second map is the bijection that sends $\cdot$ to $\lambda_{\cdot}(G)$, as described at the end of section~\ref{sec: hopfgalois skewbraces}. Since $\lambda_{\cdot_{\op}}(G)=\rho_{\cdot}(G)$, we obtain a bijection
\begin{equation*}
    \mathcal{S}\to \mathcal{R},\quad \cdot \mapsto \rho_{\cdot}(G),
\end{equation*}
which by Greither--Pareigis theory yields the equivalence of data in the statement. 
	
	We just need to show that the Hopf--Galois structures on $L/K$ can be described in the claimed  way. 
	So take an operation $\cdot$ such that $(G,\cdot,\circ)$ is a skew brace. Clearly $(G,\cdot)\cong \rho_{\cdot}(G)$, via the map
	\begin{equation*}
		\si\mapsto \rho_{\cdot}(\si). 
	\end{equation*} 
	This yields an $L$-Hopf algebra isomorphism $L[G,\cdot]\to L[\rho_{\cdot}(G)]$. Let $(G,\circ)$ act on $(G,\cdot)$ via the gamma function of $(G,\cdot,\circ)$. We show that this isomorphism is also $(G,\circ)$-equivariant. It is enough to show that for all $\si,\ta\in G$,
	\begin{equation*}
		\rho_{\cdot}(\lexp{\ta}{\gamma(\si)})=\lambda_{\circ}(\si)\rho_{\cdot}(\ta)\lambda_{\circ}(\si)^{-1}.
	\end{equation*}
	The claim follows because the left-hand side element is the unique element of $\rho_{\cdot}(G)$ which sends $1\in G$ to 
	\begin{equation*}
		(\lexp{\ta}{\gamma(\si)})^{-1}=\lexpp{\ta^{-1}}{\gamma(\si)}=\si^{-1}\cdot (\si\circ\ta^{-1}),
	\end{equation*}
	while the right-hand side element is the unique element of $\rho_{\cdot}(G)$ which sends $1\in G$ to
	\begin{equation*}
		 \si\circ(\overline{\si}\cdot \ta^{-1})=(\sigma\circ\overline{\sigma})\cdot \si^{-1}\cdot (\si\circ\ta^{-1})=\si^{-1}\cdot (\si\circ\ta^{-1}).
	\end{equation*}
	By Galois descent, we derive that $L[G,\cdot]^{(G,\circ)}$ and $L[\rho_{\cdot}(G)]^{(G,\circ)}$ are isomorphic as $K$-Hopf algebras, and the isomorphism is given as follows:
	\begin{equation*}
	\sum_{\si\in G}\ell_{\sigma}\si \mapsto  \sum_{\si\in G}\ell_{\sigma}\rho_{\cdot}(\si).
\end{equation*}
To conclude, we need to find the action of $L[G,\cdot]^{(G,\circ)}$ on $L$ that respects this isomorphism:
\begin{align*}
	\left(\sum_{\si\in G}\ell_{\sigma}\si\right) \star x&=\left(\sum_{\si\in G}\ell_{\sigma}\rho_{\cdot}(\si)\right) \star x=\sum_{\si\in G}\ell_{\sigma} (\rho_{\cdot}(\si)^{-1}[1])(x)\\
	&=\sum_{\si\in G}\ell_{\sigma} \si(x). \qedhere
\end{align*}
\end{proof}

\begin{remark}
    We believe that this point of view could simplify computation. Indeed, note the similarities of the Hopf--Galois action described in Theorem~\ref{theorem: main} with the usual Galois action. Also, an important role is played by the gamma function, which is a well-known and studied feature of a skew brace. 
\end{remark}

\begin{remark}
Following the proof of Theorem~\ref{theorem: main}, it should be clear that we are associating with a Hopf--Galois structure on $L/K$ the skew brace that is opposite to the usual one. Explicitly, given a Hopf--Galois structure in Greither--Pareigis terms, so a regular subgroup $N$ of $\Perm(G)$ normalised by $\lambda_{\circ}(G)$, then the way to find the operation $\cdot$ associated to this structure is the following: 
\begin{equation*}
    \sigma\cdot\tau=\nu(\nu^{-1}(\tau)\nu^{-1}(\sigma)),
\end{equation*}
where $\nu\colon N\to G$ is the usual bijection that maps $\eta$ to $\eta[1]$.
\end{remark}

For the rest of the section, we fix a finite Galois extension $L/K$ with Galois group $(G,\circ)$.
\begin{notation}
    To lighten the notation, we associate a Hopf--Galois structure on $L/K$ with a skew brace $(G,\cdot,\circ)$, implicitly meaning the operation $\cdot$ such that $(G,\cdot,\circ)$ is a skew brace.
\end{notation}

We immediately see that the new version of the connection fixes the peculiar behaviour described in Example~\ref{example: peculiar}. 

\begin{example}
\leavevmode
	\begin{itemize}
		\item Consider the trivial skew brace $(G,\circ,\circ)$. As the gamma function in this case is given by $\gamma(\si)=\id$, we find that the Hopf algebra in the Hopf--Galois structure on $L/K$ associated with $(G,\circ,\circ)$ is $K[G,\circ]$, and we recover the classical structure.  
		\item If instead we consider the almost trivial skew brace $(G,\circ_{\op},\circ)$, we find the Hopf--Galois  structure on $L/K$ originally corresponding to $\lambda_{\circ}(G)$, that is, the canonical nonclassical structure. 
	\end{itemize}
\end{example}

\begin{example}\label{example: example}
   Let $A$ and $B$ be finite groups. Consider a group homomorphism $\alpha\colon B\to \Aut(A)$, and suppose that $(G,\circ)$ is the semidirect product of $A$ and $B$ with respect to $\alpha$. Given $(a,b)\in G$ and $x\in L$, write $(a,b)(x)$ for the Galois action. Finally, take  $(G,\cdot)=A\times B$. Then by~\cite[Example 1.4]{GV17}, we have that $(G,\cdot,\circ)$ is a skew brace. We obtain a Hopf--Galois structure on $L/K$, which we now describe. 
    
    First, a straightforward calculation shows that the gamma function of $(G,\cdot,\circ)$ is given as follows:
    \begin{equation*}
        \lexpp{a,b}{\gamma(c,d)}=(\lexp{a}{\alpha(d)},b).
    \end{equation*}
    In particular, the $K$-Hopf algebra $L[G,\cdot]^{(G,\circ)}$ we obtain consists of the elements $\sum_{(a,b)\in G}\ell_{(a,b)}(a,b)\in L[G,\cdot]$ that satisfy, for all $(c,d)\in G$,
    \begin{equation*}
       \sum_{(a,b)\in G}\ell_{(a,b)}(a,b)= \sum_{(a,b)\in G}[(c,d)(\ell_{(a,b)})](\lexp{a}{\alpha(d)},b).
    \end{equation*}
    Such an element acts on $L$ as follows:
    \begin{equation*}
        \left(\sum_{(a,b)\in G}\ell_{(a,b)}(a,b)\right)\star x=\sum_{(a,b)\in G}\ell_{(a,b)}(a,b)(x).
    \end{equation*}
\end{example}

Given a Hopf--Galois structure on $L/K$ with associated skew brace $(G,\cdot,\circ)$, we can define the \emph{type} of the structure to be the isomorphism class of $(G,\cdot)$; note that this coincides with the usual definition. In particular, the known results about existence and classification can also be translated and obtained using the new point of view. Indeed, Theorem~\ref{theorem: quantitative} immediately follows from Theorem~\ref{theorem: main}, as well as the result counting the number of Hopf--Galois structures associated with the same isomorphism class of a skew brace. We recall this result and its proof here, which is just a slight modification of the proof of~\cite[Corollary 3.1]{KT22}. 
\begin{proposition}
	Let $(G,\cdot,\circ)$ be a skew brace. Then there are 
	\begin{equation*}
		\frac{|\Aut(G,\circ)|}{|\Aut(G,\cdot,\circ)|}
	\end{equation*}
	Hopf--Galois structures on $L/K$ such that the associated skew brace is isomorphic to $(G,\cdot,\circ)$. 
\end{proposition}
\begin{proof}
	 Consider the set $\mathcal{S}$ of group operations $\cdot'$ on $G$ such that $(G,\cdot',\circ)$ is a skew brace. We need to count for how many operations $\cdot'\in \mathcal{S}$, the skew brace $(G,\cdot',\circ)$ is isomorphic to $(G,\cdot,\circ)$. There is an action of $\Aut(G,\circ)$ on $\mathcal{S}$, as follows:
	\begin{equation*}
		\phi\colon \cdot'\to \cdot'_{\phi}, \quad \si\cdot_{\phi}'\ta=\phi(\phi^{-1}(\si)\cdot' \phi^{-1}(\ta)). 
	\end{equation*}
	Then the orbit of $\cdot\in \mathcal{S}$ consists precisely of the operations $\cdot'$ such that $(G,\cdot',\circ)$ is a skew brace isomorphic to $(G,\cdot,\circ)$. As the stabiliser of $\cdot$ under this action is $\Aut(G,\cdot,\circ)$, we derive the assertion.  
\end{proof}

We also remark that Byott's translation~\cite{Byo96} for Galois extensions, an extremely useful tool to count Hopf--Galois structures, can be obtained in this fashion. We recall here the statement and a quick proof, along the lines of the one described in~\cite[section 7]{Chi00}, but without involving regular subgroups. Let $(N,\cdot)$ be a group of the same order as $(G,\circ)$. Denote by $e(G,N)$ the number of Hopf--Galois structures on $L/K$ of type $(N,\cdot)$, which by Theorem~\ref{theorem: main} equals the number of operations $\cdot$ such that $(G,\cdot,\circ)$ is a skew brace with $(G,\cdot)\cong (N,\cdot)$, and denote by $f(G,N)$ the number of operations $\circ$ such that $(N,\cdot,\circ)$ is a skew brace with $(N,\circ)\cong (G,\circ)$.

\begin{theorem}
The following equality holds:
\begin{equation*}
    e(G,N)=\frac{|\Aut(G,\circ)|}{|\Aut(N,\cdot)|}f(G,N).
\end{equation*}
\end{theorem}

\begin{proof}
    Consider $\mathcal{N}=\{\text{bijections }\varphi\colon N\to G\}$ and $\mathcal{G}=\{\text{bijections }\psi\colon G\to N\}$. Clearly, there exists a bijection
    \begin{equation*}
        \delta\colon \mathcal{N}\to \mathcal{G},\quad \varphi\mapsto \varphi^{-1}.
    \end{equation*}
    For all $\varphi\in \mathcal{N}$, consider $(G,\cdot_{\varphi})$, where $\cdot_{\varphi}$ is the operation obtained by $\varphi$ via transport of structure. In particular, $\varphi\colon (N,\cdot)\to (G,\cdot_{\varphi})$ is an isomorphism. Similarly, for all $\psi\in \mathcal{G}$, one can define $(N,\circ_{\psi})$. It is straightforward to check that $\delta$ restricts to a bijection 
    \begin{equation*}
        \mathcal{N}'=\{\varphi\in\mathcal{N}\mid (G,\cdot_{\varphi},\circ) \text{ is a skew brace}\}\to \mathcal{G}'=\{\psi\in\mathcal{G}\mid (N,\cdot,\circ_{\psi}) \text{ is a skew brace}\}.
    \end{equation*}
    Note that the right action of $\Aut(N,\cdot)$ on $\mathcal{N}'$ via composition satisfies the following properties:
    \begin{itemize}
        \item The orbits of $\mathcal{N}'$ under the action of $\Aut(N,\cdot)$ correspond bijectively to the operations $\cdot$  such that $(G,\cdot,\circ)$ is a skew brace and $(N,\cdot)\cong (G,\cdot)$.
        \item The action of $\Aut(N,\cdot)$ on $\mathcal{N}'$ is fixed-point-free.
    \end{itemize}
    We deduce that the cardinality of $\mathcal{N}'$ equals $|\Aut(N,\cdot)|e(G,N)$. A similar argument yields that the cardinality of $\mathcal{G}'$ equals $|\Aut(G,\circ)|f(G,N)$, so 
    \begin{equation*}
    	|\Aut(N,\cdot)|e(G,N)=|\Aut(G,\circ)|f(G,N).\qedhere
    \end{equation*}
\end{proof}

We describe now the structure of the Hopf algebras in terms of the associated skew braces. Consider a Hopf--Galois structure on $L/K$, with associated skew brace $(G,\cdot,\circ)$.

\begin{theorem}
	The $K$-sub Hopf algebras of $L[G,\cdot]^{(G,\circ)}$ are precisely those of the form $L[G',\cdot]^{(G,\circ)}$ for left ideals $G'$ of $(G,\cdot,\circ)$. Moreover, $L[G',\cdot]^{(G,\circ)}$ is normal in $L[G,\cdot]^{(G,\circ)}$ if and only if $G'$ is a strong left ideal of $(G,\cdot,\circ)$. 
\end{theorem}
\begin{proof}
	This follows from Galois descent and the fact that the subgroups of $(G,\cdot)$ invariant under the action of $(G,\circ)$ via the gamma function of $(G,\cdot,\circ)$ are precisely the left ideals of $(G,\cdot,\circ)$. 
\end{proof}

Consider a left ideal $G'$ of $(G,\cdot,\circ)$. Then $G'$ corresponds to an intermediate field $L^{H'}$ of $L/K$ via the Hopf--Galois correspondence, where $H'=L[G',\cdot]^{(G,\circ)}$. But as $G'$ is a subgroup of $(G,\circ)$, we have that $G'$ also corresponds to an intermediate field $F$ of $L/K$ via the usual Galois correspondence. We denote both fields by $L^{G'}$, the ambiguity justified by the following pleasant consequence of Theorem~\ref{theorem: main}.
\begin{corollary}\label{corollary: nice equality}
The following equality holds:
	\begin{equation*}
		L^{H'}=F
	\end{equation*}
\end{corollary}

\begin{proof}
	 It is clear that if $x\in F$, then $x\in L^{H'}$. Indeed, given $\sum_{\sigma\in G}\ell_{\sigma} \sigma\in H'$, we have 
	\begin{equation*}
		\left(\sum_{\si\in G}\ell_{\sigma} \si\right)\star x=\sum_{\si\in G} \ell_{\sigma}\si(x)=\sum_{\si\in G} \ell_{\sigma} x=\ep\left(\sum_{\si\in G}\ell_{\sigma} \si\right)x.
	\end{equation*}
	The assertion then follows from $[L:F]=|G'|=[L:L^{H'}]$.
\end{proof}

As the action of $(G,\circ)$ on $(G,\cdot)$ is given by the gamma function of $(G,\cdot,\circ)$, we can easily describe the grouplike elements of $L[G,\cdot]^{(G,\circ)}$.
\begin{corollary}
The grouplike elements of the $K$-Hopf algebra $L[G,\cdot]^{(G,\circ)}$ are the elements of $\Fix(G,\cdot,\circ)$.
\end{corollary}

We study now how several known notions in skew brace theory have a natural description in Hopf--Galois theory.

 \textbf{Left ideals.} As already mentioned, a left ideal $G'$ of $(G,\cdot,\circ)$ corresponds to a $K$-sub Hopf algebra $L[G',\cdot]^{(G,\circ)}$ of $L[G,\cdot]^{(G,\circ)}$, which then corresponds to an intermediate field $F=L^{G'}$ of $L/K$. The extension $L/F$ is Galois with Galois group $(G',\circ)$, and there exists a natural Hopf--Galois structure on $L/F$ given by the $F$-Hopf algebra $F\otimes_{K} L[G',\cdot]^{(G,\circ)}$. The skew brace associated with this Hopf--Galois structure is precisely $(G',\cdot,\circ)$. Indeed, by Galois descent, the natural map
 \begin{equation*}
 	F\otimes_{K} L[G',\cdot]^{(G,\circ)}\to L[G',\cdot]^{(G',\circ)}
 \end{equation*}
 is an $F$-Hopf algebra isomorphism, and as both the actions of these Hopf algebras on $L$ are obtained by that of $L[G,\cdot]^{(G,\circ)}$, the assertion easily follows. 
 
\textbf{Strong left ideals.}
Suppose in addition that $G'$ is a strong left ideal of $(G,\cdot,\circ)$, so $G'$ is normal in $(G,\cdot)$. In this case, $L[G',\cdot]^{(G,\circ)}$ is normal in $L[G,\cdot]^{(G,\circ)}$, and we obtain a short exact sequence of  $K$-Hopf algebras
	\begin{equation*}
		K\to L[G',\cdot]^{(G,\circ)}\to L[G,\cdot]^{(G,\circ)}\to L[G/G',\cdot]^{(G,\circ)}\to K.
	\end{equation*}
	We find a Hopf--Galois structure on $F/K$ with $K$-Hopf algebra  $L[G/G',\cdot]^{(G,\circ)}$.
	
\textbf{Ideals.} Finally, suppose that $G'$ is an ideal of $(G,\cdot,\circ)$.
	Then $F/K$ is Galois with Galois group $(G/G',\cdot)$, and the Hopf--Galois structure on $F/K$ given by $L[G/G',\cdot]^{(G,\circ)}$ is associated with the skew brace $(G/G',\cdot,\circ)$, because in this case the equality $L[G/G',\cdot]^{(G,\circ)}=F[G/G',\cdot]^{(G/G',\circ)}$ holds.

\textbf{Semidirect products.} Suppose that $(G,\cdot,\circ)$ is isomorphic to a semidirect product of skew braces. Then there exists an ideal $G_1$ and a strong left ideal $G_2$ of $(G,\cdot,\circ)$ such that $(G,\circ)$ is the inner semidirect product of $(G_1,\circ)$ and $(G_2,\circ)$, and $(G,\cdot)$ is the inner direct product of $(G_1,\cdot)$ and $(G_2,\cdot)$. Write $F_1=L^{G_1}$ and $F_2=L^{G_2}$. In this case, the towers $K\subseteq F_1\subseteq L$ and $K\subseteq F_2\subseteq L$ are described exactly as before. Moreover, $L[G,\cdot]$ is isomorphic to $L[G_1,\cdot]\otimes_{L}L[G_2,\cdot]$ as $(G,\circ)$-compatible $L$-Hopf algebras, and by Galois descent, 
	\begin{equation*}
		L[G,\cdot]^{(G,\circ)}\cong L[G_1,\cdot]^{(G,\circ)}\otimes_{K}L[G_2,\cdot]^{(G,\circ)}
	\end{equation*} as $K$-Hopf algebras.
	
	We note moreover that because $G_1$ is an ideal of $(G,\cdot,\circ)$, the obvious isomorphism
    $\ph\colon (G_2,\circ)\to (G/G_1,\circ)$ between Galois groups is in fact an isomorphism of skew braces
    $\ph\colon (G_2,\cdot,\circ)\to (G/G_1,\cdot,\circ)$. This implies that the Hopf--Galois structures on $L/F_2$ and $F_1/K$ given by the previous description are associated with  skew braces isomorphic in a natural way. By this observation and Galois descent, we can also deduce that
	\begin{equation*}
		F_2\otimes_K F_1[G/G_1,\cdot]^{(G/G_1,\circ)}\cong L[G_2,\cdot]^{(G_2,\circ)}
	\end{equation*}
	as $F_2$-Hopf algebras. 
	
\textbf{Direct products.} If the semidirect product is also direct, then the Galois group $(G,\circ)$ is the inner direct product of $(G_1,\circ)$ and $(G_2,\circ)$, and we can repeat the previous analysis also for $F_2/K$, which is Galois in this case.

\textbf{Metatriviality.} Suppose now that $(G,\cdot,\circ)$ metatrivial. Consider an ideal $G'$ of $(G,\cdot,\circ)$ such that $(G',\cdot,\circ)$ and $(G/G',\cdot,\circ)$ are trivial skew braces, and write $F=L^{G'}$.  Then the Hopf--Galois structures on $L/F$ and $F/K$ obtained by the action of $L[G,\cdot]^{(G,\circ)}$ on $L$ are the classical structures.

\begin{remark}
   There are notions of solubility and nilpotency of skew braces that generalise metatriviality; see, for example,~\cite{CSV19,KSV21}. For Hopf--Galois structures associated with skew braces $(G,\cdot,\circ)$ with these properties, similar conclusions, involving tower of intermediate fields of $L/K$, can be derived.

	 We believe that the study of this kind of properties of the skew brace, together with appropriate ramification on the extension $L/K$, could bring new results in Hopf--Galois module theory. For example, in~\cite{Byo02}, a key role for the study of the Hopf--Galois module structure of a Galois extension $L/K$ of $p$-adic fields of degree $p^2$, $p$ a prime, was played by an intermediate normal field $F$ of $L/K$ such that, given a Hopf--Galois structure on $L/K$, $F$ is in the image of the Hopf--Galois correspondence, and the Hopf--Galois structure on $L/K$ yields the classical structures on $L/F$ and $F/K$.
   
   More generally, all the skew braces obtained with~\cite[Theorem 6.6]{ST23}, which generalise several constructions developed in recent years, are metatrivial, so similar reasonings could be repeated. 
\end{remark}

\section{The Hopf--Galois correspondence}\label{sec: examples}
In this final section, we study the Hopf--Galois correspondence with respect to the new version of the connection. 
We fix a finite Galois extension of fields $L/K$ with Galois group $(G,\circ)$. From the discussion of section~\ref{sec: main}, we immediately derive the following result.
\begin{corollary}\label{corollary: main description}
	Consider a Hopf--Galois structure on $L/K$, with associated skew brace $(G,\cdot,\circ)$. Then the Hopf--Galois correspondence for this structure is surjective if and only if every subgroup of $(G,\circ)$ is a left ideal of $(G,\cdot,\circ)$. 
	
	Specifically, if $G'$ is a subgroup of $(G,\circ)$, then $L^{G'}$ is in the image of the Hopf--Galois correspondence if and only if $G'$ is a left ideal of $(G,\cdot,\circ)$. 
\end{corollary}

\begin{example}
	Consider the classical structure, with associated skew brace $(G,\circ,\circ)$. In this case, every subgroup of $(G,\circ)$ is a left ideal of $(G,\circ,\circ)$, so we find, as expected, that the Hopf--Galois correspondence for this structure is surjective.
\end{example}

We note the following facts, which are direct consequences of Corollary~\ref{corollary: main description}

\begin{remark}\label{remark: not enough subgroups}
	If $(G,\cdot,\circ)$ is a skew brace and $(G,\cdot)$ has less subgroups than $(G,\circ)$, then for the Hopf--Galois structure on $L/K$ associated with $(G,\cdot,\circ)$, the Hopf--Galois correspondence is not surjective. 
\end{remark}

\begin{remark}\label{remark: direct product}
Suppose that $(G,\cdot,\circ)$ is a skew brace isomorphic to the direct product of skew braces $(G_i,\cdot,\circ)$ of pairwise coprime orders. If all the subgroups of $(G_i,\circ)$ are left ideals of $(G_i,\cdot,\circ)$, then all the subgroups of $(G,\circ)$ are left ideals of $(G,\cdot,\circ)$, so for the Hopf--Galois structure on $L/K$ associated with $(G,\cdot,\circ)$, the Hopf--Galois correspondence is surjective.
\end{remark}

We focus now our attention on Hopf--Galois structures associated with bi-skew braces. In this case, the gamma functions take values in the automorphisms of the Galois group $(G, \circ)$, so we easily derive the following fact.
\begin{lemma}
Consider a Hopf--Galois structure on $L/K$ such that the associated skew brace $(G,\cdot,\circ)$ is a bi-skew brace. Let $G'$ be a characteristic subgroup of $(G,\circ)$. Then $L^{G'}$ is in the image of the Hopf--Galois correspondence for this structure.
\end{lemma}

\begin{corollary}\label{corollary: cyclic}
	Suppose that $(G,\circ)$ is a cyclic group, and consider a Hopf--Galois structure on $L/K$ such that the associated skew brace $(G,\cdot,\circ)$ is a bi-skew brace.
	 Then the Hopf--Galois correspondence for this structure is surjective. 
\end{corollary}

\begin{example}
	Suppose that $(G,\circ)$ is cyclic of order $8$. As shown in~\cite{Rum07b}, there exists a skew brace $(G,\circ,\cdot)$ with $(G,\cdot)\cong Q_8$, the quaternion group. A straightforward calculation shows that $(G,\circ,\cdot)$ is a bi-skew brace. We conclude by Corollary~\ref{corollary: cyclic} that for the Hopf--Galois structure on $L/K$ associated with the skew brace $(G,\cdot,\circ)$, the Hopf--Galois correspondence is surjective.  
\end{example}

We remark that for a Hopf--Galois structure on $L/K$ associated with a bi-skew brace $(G,\cdot,\circ)$, the Hopf--Galois correspondence is surjective if and only if $\gamma(\si)$ is a \emph{power automorphism} of $(G,\circ)$ for all  $\si\in G$, that is, $\lexp{\ta}{\gamma(\si)}$ is a power of $\ta$ in $(G,\circ)$ for all $\ta\in G$. Indeed, the power automorphisms of $(G,\circ)$ are precisely the automorphisms of $(G,\circ)$ that map every subgroup of $(G,\circ)$ to itself. 
\begin{example}\label{example: directsemidirect}
Suppose that $(G,\circ)$ is the direct product of an abelian group $A$ and the cyclic group $C_2$ of order $2$. Denote by $\alpha$ the action of $C_2$ on $A$ via inversion, and consider the semidirect product  $(G,\cdot)=A\rtimes C_2$ with respect to this action. Then $(G,\cdot,\circ)$ is a bi-skew brace; see~\cite[Examples 1.4 and 1.5]{GV17}. Here the gamma function of $(G,\cdot,\circ)$ is given as follows:
	\begin{equation*}
		\lexpp{a,b}{\gamma(c,d)}=(\lexpp{a}{\alpha(d^{-1})},b),
	\end{equation*}
	which is either equal to $(a,b)$ or to $\overline{(a,b)}$. In particular, $\gamma(c,d)$ is a power automorphism of $(G,\circ)$, and we conclude that for the Hopf--Galois structure on $L/K$ associated with $(G,\cdot,\circ)$, the Hopf--Galois correspondence is surjective. 
\end{example}

We deal now with bi-skew braces $(G,\cdot,\circ)$ whose gamma functions have values in the inner automorphism group of $(G,\circ)$; these skew braces have been recently studied in~\cite{Koc21,CS21,Koc22,CS22a-p,ST23}.
Denote by $Z(G)$ the centre of $(G,\circ)$ and by $N(G)$ the \emph{norm} of $(G,\circ)$, that is, the intersection of the normalisers of the subgroups of $(G,\circ)$. It is clear that $\iota_{\circ}(\si)$ is a power automorphism of $(G,\circ)$ if and only if $\si\in N(G)$.

We can apply this fact to obtain Hopf--Galois structures on $L/K$ for which the Hopf--Galois correspondence is surjective, as follows. Given a group homomorphism $\psi\colon (G,\circ)\to N(G)/Z(G)$, define \begin{equation*}
		\si\cdot_{\psi}\ta=\si\circ\lexp{\ta}{\iota_{\circ}(\psi(\si))}=\si\circ\psi(\si)\circ \ta\circ\overline{\psi(\si)};
	\end{equation*}
	here by $\psi(\si)$ we denote any element in the coset $\psi(\si)$ in $N(G)/Z(G)$, with a little abuse of notation justified by the fact that if $\ta\in Z(G)$, then $\iota_{\circ}(\ta)=\id$.
\begin{theorem}\label{theorem: norm}
	For all group homomorphisms $\psi\colon (G,\circ)\to N(G)/Z(G)$, we have that  $(G,\cdot_{\psi},\circ)$ is a bi-skew brace, and for the Hopf--Galois structure on $L/K$ associated with $(G,\cdot_{\psi},\circ)$, the Hopf--Galois correspondence is surjective. 
\end{theorem} 

\begin{proof}
	Let $\psi\colon (G,\circ)\to N(G)/Z(G)$ be a group homomorphism. By the main theorem of~\cite{Sch60}, the quotient $N(G)/Z(G)$ is abelian, so we can apply~\cite[Theorem 6.6]{ST23} to derive that $(G,\cdot_{\psi},\circ)$ is a bi-skew brace and the gamma function of $(G,\cdot_{\psi},\circ)$ is given by $\gamma(\si)=\iota_{\circ}(\overline{\psi(\si)})$.  In particular, the gamma function of $(G,\cdot_{\psi},\circ)$ is given by conjugation by elements of $N(G)$ in $(G,\circ)$, so by power automorphisms of $(G,\circ)$, and therefore we obtain our assertion. 
\end{proof}
\begin{remark}
	Note that distinct group homomorphisms $(G,\circ)\to N(G)/Z(G)$ in Theorem~\ref{theorem: norm} yield distinct operations, so distinct Hopf--Galois structures on $L/K$. 
\end{remark}

\begin{example}\label{example: quaternion}
	Suppose that $(G,\circ)=Q_8$, the quaternion group of order $8$. There are 22 Hopf--Galois structures on $L/K$, and 6 of them are of cyclic type; see~\cite[Table 2]{SV18}. As $(G,\circ)$ is Hamiltonian, we derive that $N(G)=G$, so $N(G)/Z(G)\cong C_2\times C_2$. Since there are $16$ distinct group homomorphisms
	\begin{equation*}
		Q_8\to C_2\times C_2,
	\end{equation*} we obtain $16$ distinct Hopf--Galois structures on $L/K$ for which the Hopf--Galois correspondence is surjective. We find indeed all the Hopf--Galois structures on $L/K$ except for the $6$ of cyclic type, for which the Hopf--Galois correspondence is not surjective by Remark~\ref{remark: not enough subgroups}. 
\end{example}

\begin{example}
	Suppose that $(G,\circ)$ is the nonabelian group of order $p^3$ and exponent $p^2$, with $p$ odd prime. An easy reasoning implies that $N(G)$ is the elementary abelian subgroup of $(G,\circ)$ of order $p^2$, while the centre is cyclic of order $p$. As there are $p^2$ distinct group homomorphisms
	\begin{equation*}
		(G,\circ)\to C_p,
	\end{equation*}
	we obtain $p^2$ distinct Hopf--Galois structures on $L/K$ for which the Hopf--Galois correspondence is surjective.
\end{example}

The following result, whose proof is immediate, shows that the behaviour of the canonical nonclassical structure can also be assumed by other Hopf--Galois structures. 
\begin{proposition}\label{proposition: canonical nonclassical}
	Consider a Hopf--Galois structure on $L/K$ such that associated skew brace $(G,\cdot,\circ)$ is a bi-skew brace with gamma function $\gamma\colon (G,\circ)\to \Inn(G,\circ)$. Then every normal intermediate field $K$ of $L/K$ is in the image of the Hopf--Galois correspondence for this structure.
	
	Moreover, if $\gamma\colon (G,\circ) \to \Inn(G,\circ)$ is surjective, then the image of the Hopf--Galois correspondence consists precisely of the normal intermediate fields of $L/K$. 
\end{proposition}

\begin{example}
	Consider the canonical nonclassical structure, with associated skew brace $(G,\circ_{\op},\circ)$. Here $\gamma(\si)=\iota_{\circ}(\si)$ for all $\si\in G$. Applying Proposition~\ref{proposition: canonical nonclassical}, we recover 
	 the well-known property of the canonical nonclassical structure. 
\end{example}

\begin{example}
	Suppose that $(G,\circ)$ is nilpotent of class two, and define
	\begin{equation*}
		\si\cdot \ta=\si\circ \lexp{\ta}{\iota_{\circ}(\si)} =\si\circ \si\circ \ta\circ \overline\si.
	\end{equation*}
	Then by~\cite[Proposition 5.6]{CS22a-p}, we have that $(G,\cdot,\circ)$ is a bi-skew brace and the gamma function of $(G,\cdot,\circ)$ is given by $\gamma(\si)=\iota_{\circ}(\overline{\si})$. By Proposition~\ref{proposition: canonical nonclassical}, we derive that for the associated Hopf--Galois structure on $L/K$, the image of the Hopf--Galois correspondence consists precisely of the normal intermediate fields of $L/K$.
	
	It is easy to see that if there exists $\si\in G$ such that $\si\circ\si$ is not in the centre of $(G,\circ)$, then the Hopf--Galois structure we find is different from the canonical nonclassical structure. This holds, for examples, for the Heisenberg group of order $p^3$, with $p$ an odd prime. 
\end{example}

We study now a question posed in~\cite{Chi21}. Let $L_1/K_1$ be a finite Galois extension of fields with Galois group $(G,\circ)$, and consider a Hopf--Galois structure on $L_1/K_1$, with associated skew brace $(G,\cdot,\circ)$. We can rewrite the \emph{Hopf--Galois correspondence ratio}, defined as the ratio of the number intermediate fields of $L_1/K_1$ in the image of the Hopf--Galois correspondence to the number of intermediate fields of $L_1/K_1$, as follows:

\begin{equation*}
	GC(L_1/K_1,L_1[G,\cdot]^{(G,\circ)})=\frac{|\{\text{left ideals of $(G,\cdot,\circ)$}\}|}{|\{\text{subgroups of $(G,\circ)$}\}|}.
\end{equation*}

Suppose in addition that $(G,\cdot,\circ)$ is a bi-skew brace, and let $L_2/K_2$ be a finite Galois extension of fields with Galois group $(G,\cdot)$. The skew brace $(G,\circ,\cdot)$ is associated with a Hopf--Galois structure on $L_2/K_2$. Are these two Hopf--Galois structures related in some way?

The next result follows immediately from the facts that the lattices of left ideals of $(G,\cdot,\circ)$ and $K_1$-sub Hopf algebras of $L_1[G,\cdot]^{(G,\circ)}$ are isomorphic, and the left ideals of $(G,\cdot,\circ)$ and $(G,\circ,\cdot)$ coincide.
\begin{theorem}
The following facts hold:
	\begin{itemize}
	    \item The lattices of $K_1$-sub Hopf algebras of $L_1[G,\cdot]^{(G,\circ)}$ and $K_2$-sub Hopf algebras of $L_2[G,\circ]^{(G,\cdot)}$ are isomorphic.
		\item There is the same number of intermediate fields in the images of the Hopf--Galois correspondence for the Hopf--Galois structure on $L_1/K_1$ associated with $(G,\cdot,\circ)$ and the Hopf--Galois structure on $L_2/K_2$ associated with $(G,\circ,\cdot)$.
		\item The following equality holds: 
				\begin{equation*}
					\frac{GC(L_1/K_1,L_1[G,\cdot]^{(G,\circ)})}{GC(L_2/K_2,L_2[G,\circ]^{(G,\cdot)})}=\frac{|\{\text{subgroups of $(G,\cdot)$}\}|}{|\{\text{subgroups of $(G,\circ)$}\}|}.
				\end{equation*} 
				In particular, the ratio between the two Hopf--Galois correspondence ratios is constant and depends only on the isomorphism classes of the Galois groups. 
	\end{itemize}
\end{theorem} 

\begin{example}
	Suppose that $(G,\cdot,\circ)$ is the skew brace of Example~\ref{example: directsemidirect} with $p$ an odd prime and $A=C_p$. Then $(G,\cdot)$ is dihedral of order $2p$ and $(G,\circ)$ is cyclic of order $2p$. There are  $p+3$ subgroups of $(G,\cdot)$ and $4$ subgroups of $(G,\circ)$, and as every subgroup of $(G,\circ)$ is a left ideal of $(G,\cdot,\circ)$, we have the following equalities: 
	\begin{align*}
	GC(L_1/K_1,L_1[G,\cdot]^{(G,\circ)})&=1,\\
	GC(L_2/K_2,L_2[G,\circ]^{(G,\cdot)})&=\frac{4}{p+3},\\
		\frac{GC(L_1/K_1,L_1[G,\cdot]^{(G,\circ)})}{GC(L_2/K_2,L_2[G,\circ]^{(G,\cdot)})}&=\frac{p+3}{4}.
	\end{align*}
\end{example}

We conclude by focusing our attention on Hopf--Galois structures associated with skew braces that are not necessarily bi-skew braces. We begin with the following theorem, which was proved in~\cite{Koh19}. We provide a quick proof for convenience.

\begin{theorem}\label{theorem: kohl}
	Let $N$ be a group. If there exists $m$ such that the number of characteristic subgroups of order $m$ of $N$ is greater than the number of subgroups of order $m$ of $(G,\circ)$, then $L/K$ has no Hopf--Galois structures of type $N$.  
\end{theorem}

\begin{proof}
	If $L/K$ has a Hopf--Galois structure of type $N$, then there exists a skew brace $(G,\cdot,\circ)$ with $(G,\cdot)\cong N$. As every characteristic subgroup of $(G,\cdot)$ is a left ideal of $(G,\cdot,\circ)$, so also a subgroup of $(G,\circ)$, we immediately derive a contradiction.
\end{proof}

On the contrary, if there exists a skew brace $(G,\cdot,\circ)$ such that the number of characteristic subgroups of $(G,\cdot)$ equals the number of subgroups of $(G,\circ)$, then the Hopf--Galois structure on $L/K$ associated with $(G,\cdot,\circ)$ assumes a nice behaviour. 
\begin{proposition}\label{proposition: same numbers}
Consider a Hopf--Galois structure on $L/K$, with associated skew brace $(G,\cdot,\circ)$. 
Suppose that the number of characteristic subgroups of $(G,\cdot)$ equals the number of subgroups of $(G,\circ)$. Then the Hopf--Galois correspondence for this structure is surjective. 
\end{proposition}
\begin{proof}
	Every characteristic subgroup of $(G,\cdot)$ is a left ideal of $(G,\cdot,\circ)$, so also a subgroup of $(G,\circ)$. In particular, every subgroup of $(G,\circ)$ is a left ideal. 
\end{proof}

\begin{example}\label{example: cyclic odd}
	Suppose that $(G,\circ)$ is cyclic of odd prime power order, and consider a Hopf--Galois structure on $L/K$, with associated skew brace $(G,\cdot,\circ)$. By~\cite{Koh98}, also $(G,\cdot)$ is cyclic, so by Proposition~\ref{proposition: same numbers}, we conclude that the Hopf--Galois correspondence is surjective; we have recovered~\cite[Proposition 4.3]{Chi17}. 
\end{example}

\begin{example}\label{example: cyclic even}
	Suppose that $(G,\circ)$ is cyclic of order $2^m$, with $m\ge 1$, and consider a Hopf--Galois structure on $L/K$, with associated skew brace $(G,\cdot,\circ)$. We claim that the Hopf--Galois correspondence for this structure is surjective. 
	
	If $m=1, 2$, then by the explicit classification in~\cite[Proposition 2.4]{Bac15}, one can check that $(G,\cdot,\circ)$ is a bi-skew brace, so the result follows from Corollary~\ref{corollary: cyclic}. 
	
	Suppose now that $m\ge 3$. By~\cite[Theorem 6.1]{Byo07}, necessarily $(G,\cdot)$ is cyclic, the dihedral group, or the generalised quaternion group. With the unique exception of $m=3$ and $(G,\cdot)\cong Q_8$,  the numbers of characteristic subgroups of $(G,\cdot)$ and subgroups of $(G,\circ)$ coincide, so we conclude by Proposition~\ref{proposition: same numbers}.
	
	Finally, suppose that $m=3$ and $(G,\cdot)\cong Q_8$. Then the centre $Z$ of $(G,\cdot)$ is a characteristic subgroup of order $2$. It follows that $Z$ is an ideal of $(G,\cdot,\circ)$. By the case $m=2$, we know that $(G/Z,\cdot,\circ)$ has a left ideal $G'/Z$ of order $2$, which easily implies that $G'$ is a left ideal of $(G,\cdot,\circ)$ of order $4$.
\end{example}

\begin{remark}
	With the classification given in~\cite{Bac15}, it is easy to construct a skew brace $(G,\cdot,\circ)$ with $(G,\circ)$ cyclic of order $p^3$, where $p$ is a prime,  such that $(G,\cdot,\circ)$ is not a bi-skew brace. Thus Examples~\ref{example: cyclic odd} and~\ref{example: cyclic even} do not follow from Corollary~\ref{corollary: cyclic}. 
\end{remark}

We shall now conclude by characterising all the Galois extensions that behave like Examples~\ref{example: cyclic odd} and~\ref{example: cyclic even}. First, a useful lemma.
\begin{lemma}\label{lemma: not surjective}
	Suppose that $(G,\circ)$ is isomorphic to a direct product of groups $(A,\circ)$ and $(B,\circ)$, and that there exists a skew brace $(A,\cdot,\circ)$ such that not every subgroup of $(A,\circ)$ is a left ideal of $(A,\cdot,\circ)$. Then there exists a Hopf--Galois structure on $L/K$ for which the Hopf--Galois correspondence is not surjective.
\end{lemma}

\begin{proof}
	We can use the group isomorphism $(G,\circ)\cong (A,\circ)\times (B,\circ)$ to transport the structure of $(A,\cdot)\times (B,\circ)$ to $G$. We obtain a group operation $\cdot$ such that $(G,\cdot,\circ)$ is a skew brace isomorphic to $(A,\cdot,\circ)\times (B,\circ,\circ)$. By assumption, there exists a subgroup of $(G,\circ)$ which is not a left ideal of $(G,\cdot,\circ)$, so for the Hopf--Galois structure on $L/K$ associated with $(G,\cdot,\circ)$, the Hopf--Galois correspondence is not surjective.
\end{proof}

\begin{theorem}\label{theorem: childs}
	The following are equivalent:
	\begin{itemize}
		\item For all the Hopf--Galois structures on $L/K$, the Hopf--Galois correspondence is surjective.
		\item The Galois group $(G,\circ)$ is cyclic, and for all primes $p$ and $q$ dividing the order of $(G,\circ)$, we have that $p$ does not divide $q-1$.  
	\end{itemize}
\end{theorem}

\begin{proof}

Suppose first that $(G,\circ)$ is cyclic of order $n$ and for all primes $p$ and $q$ dividing $n$, we have that $p$ does not divide $q-1$. Consider a Hopf--Galois structure on $L/K$, with associated skew brace $(G,\cdot,\circ)$. If $n$ is even, then $n$ has no odd prime divisors, so the result follows from Example~\ref{example: cyclic even}. 

If instead $n$ is odd, then by~\cite[Corollary 1.7]{Tsa21}, we have that $(G,\cdot)$ is isomorphic to a semidirect product of cyclic groups $C_a\rtimes C_b$, where $a$ and $b$ are coprime and $ab=n$. But by the assumption on the divisors of the order of $(G,\circ)$, this semidirect product is necessarily a direct product. In particular, $(G,\cdot)$ is cyclic, and we can apply~\cite[Corollary 4.3]{CSV19} to deduce that $(G,\cdot,\circ)$ is isomorphic to a direct product of skew braces of coprime odd prime power order. The assertion then follows from  Remark~\ref{remark: direct product} and Example~\ref{example: cyclic odd}.

Conversely, suppose that for all the Hopf--Galois structures on $L/K$, the Hopf--Galois correspondence is surjective. As this holds for the canonical nonclassical structure, $(G,\circ)$ is either abelian or Hamiltonian. We proceed by exclusion. 

 Suppose first that $(G,\circ)$ is Hamiltonian. Then there exists an abelian group $A$ such that $(G,\circ)$ is isomorphic to the direct product of $Q_8$ and $A$; see~\cite[Theorem 12.5.4]{Hal59}. As already mentioned, there exists a skew brace $(G',\cdot,\circ)$ where $(G',\circ)\cong Q_8$ and $(G',\cdot)$ is cyclic. 
 By applying Remark~\ref{remark: not enough subgroups} and Lemma~\ref{lemma: not surjective}, we derive a contradiction.  

We deduce that $(G,\circ)$ is necessarily abelian. Suppose that $(G,\circ)$ is not cyclic. Then there exists a prime $p$ such that $(G,\circ)$ is isomorphic to a direct product of the form  $C_{p^{r}}\times C_{p^{s}}\times A$, where $1\le  s \le  r$.  Write $\si$ for a generator of $C_{p^{r}}$ and $\ta$ for a generator of  $C_{p^{s}}$. In a slight variation of~\cite[Example 6.7]{ST23}, there exists a skew brace $(G',\cdot,\circ)$ such that $(G',\circ)$ equals $C_{p^{r}}\times C_{p^{s}}$ with the direct product operation and
\begin{equation*}
	(\si^i,\ta^j)\cdot (\si^a,\ta^b)=(\si^{i+a},\ta^{j+b+ia}). 
\end{equation*}
Note that the subgroup $C_{p^r}\times \{1\}$ of $(G',\circ)$ is not a subgroup of $(G',\cdot)$, so in particular it is not a left ideal of $(G',\cdot,\circ)$. Again by Lemma~\ref{lemma: not surjective}, we find a contradiction.

We deduce that $(G,\circ)$ is necessarily cyclic. Suppose that there exist primes $p$ and $q$ dividing the order of $(G,\circ)$ such that $p$ divides $q-1$. Let $(G',\circ)$ be the direct product of the Sylow $q$-subgroup $Q$ and the Sylow $p$-subgroup $P$ of $(G,\circ)$. By assumption on $p$ and $q$, we can construct a nontrivial semidirect product $(G',\cdot)$ of $Q$ and $P$. By~\cite[Example 1.5]{GV17}, we have that  $(G',\cdot,\circ)$ is a skew brace. Suppose that $\{1\}\times P$ is a left ideal of $(G',\cdot,\circ)$. Then $\{1\}\times P$ is not a left ideal of $(G',\cdot_{\op},\circ)$, because otherwise $\{1\}\times P$ would be normal subgroup of $(G',\cdot)$. As $(G,\circ)$ is isomorphic to the direct product of all its Sylow subgroups, we find a contradiction from Lemma~\ref{lemma: not surjective}.
\end{proof}

\bibliographystyle{amsalpha}
\bibliography{bib.bib}
\end{document}